\documentclass[11pt,reqno]{amsart}
\usepackage[notref,notcite]{}
\usepackage{hyperref}\hypersetup{colorlinks=true, citecolor=blue}
\usepackage{graphicx}
\usepackage{xfrac}
\usepackage{esint,amsfonts,amsmath,amssymb,epsfig,
mathrsfs,bm,accents,yfonts,mathtools}
\numberwithin{equation}{section}
\newtheoremstyle{mytheorem}
{}
{}
{\it}
{\parindent}
{\bf}
{.}
{ }
{\thmnumber{#2.~}\thmname{#1}\thmnote{~\rm#3}}

\newtheoremstyle{myremark}
{}
{}
{\rm}
{\parindent}
{\bf}
{.}
{ }
{\thmnumber{#2.~}\thmname{#1}\thmnote{~\rm#3}}

\newtheoremstyle{myparagraph}
{}
{}
{\rm}
{\parindent}
{\bf}
{}
{ }
{\thmnumber{#2.~}\thmname{#1}\thmnote{#3}}

\theoremstyle{mytheorem}

\newtheorem{theorem}[subsection]{Theorem}
\newtheorem{lemma}[subsection]{Lemma}

\newtheorem{proposition}[subsection]{Proposition}

\theoremstyle{myremark}

\newtheorem*{remark*}{Remark}

\theoremstyle{myparagraph}

\newtheorem*{parag*}{}

\makeatletter
\def\@secnumfont{\sc}
\def\section{\@startsection{section}{1}%
\z@{1.5\linespacing\@plus .2\linespacing}{.7\linespacing}%
{\normalfont\sc\centering}}
\makeatother

\makeatletter
\def\ps@headings{\ps@empty
 \def\@evenhead{%
  \setTrue{runhead}%
  \normalfont\footnotesize
  \rlap{\thepage}\hfil
  \def\thanks{\protect\thanks@warning}%
  \leftmark{}{}\hfil}%
 \def\@oddhead{%
  \setTrue{runhead}%
  \normalfont\footnotesize\hfil
  \def\thanks{\protect\thanks@warning}%
  \rightmark{}{}\hfil \llap{\thepage}}%
\let\@mkboth\markboth}
\makeatother
\makeatletter
\renewenvironment{proof}[1][\proofname]{\par
  \pushQED{\qed}%
  \normalfont \topsep6\p@\@plus6\p@\relax
  \trivlist
  \itemindent\normalparindent
  \item[\hskip\labelsep
    \scshape
    #1\@addpunct{.}]\ignorespaces
}{%
  \popQED\endtrivlist\@endpefalse
}
\providecommand{\proofname}{Proof}
\makeatother
\newcommand{\R}{\mathbb{R}}

\newcommand\res{\mathop{\hbox{\vrule height 7pt width .5pt depth 0pt
\vrule height .5pt width 6pt depth 0pt}}\nolimits}

\newcommand{\cH}{{\mathcal{H}}}

\newcommand\F{{\mathbb F}}
\newcommand\N{{\mathbb N}}

\newcommand\M{{\mathbf M}}

\begin{document}

	%
\pagestyle{empty}
\pagestyle{myheadings}
\markboth%
{\underline{\centerline{\hfill\footnotesize%
\vphantom{,}\hfill}}}%
{\underline{\centerline{\hfill\footnotesize%
\textsc{Stable extremal submanifolds}%
\vphantom{,}\hfill}}}

	%
\thispagestyle{empty}

~\vskip -1.1 cm

	%

\vspace{1.7 cm}

	%
{\Large\sl\centering
Quantitative minimality of strictly stable extremal\\ submanifolds in a flat neighbourhood
\\
}

\vspace{.4 cm}

	%
\centerline{\sc Dominik Inauen \& Andrea Marchese}

\vspace{.8 cm}

{\rightskip 1 cm
\leftskip 1 cm
\footnotesize
{\sc Abstract.}
In this paper we extend the results of \textit{A strong minimax property of nondegenerate minimal submanifolds}, by White, where it is proved that any smooth, compact submanifold, which is a strictly stable critical point for an elliptic parametric functional, is the unique minimizer in a certain geodesic tubular neighbourhood. We prove a similar result, replacing the tubular neighbourhood with one induced by the flat distance and we provide quantitative estimates. Our proof is based on the introduction of a penalized minimization problem, in the spirit of \textit{A selection principle for the sharp quantitative isoperimetric inequality}, by Cicalese and Leonardi, which allows us to exploit the regularity theory for almost minimizers of elliptic parametric integrands.

\vspace{4pt}

\vspace{4pt}

}

%
%
\section{Introduction}
It is well known that any strictly stable critical point of a smooth function $f:\R^n\to\R$ is locally its unique minimizer. In \cite{Wh}, B. White proves a statement of similar nature in a space of submanifolds of a Riemannian manifold, where the function $f$ is replaced by an elliptic parametric functional. In his setting the term ``locally'' above should be intended with respect to the \textit{strong} topology induced by the Riemannian distance. In the present paper we improve such result, replacing the strong topology with the one induced by the flat distance and providing also quantitative estimates. More precisely we prove the following result, where by $\mathbb{F}(T)$ we denote the flat norm of the integral current $T$. 

\begin{theorem}\label{th1_w}
Let $M^{m}$ be a smooth, compact Riemannian manifold (or $M =  \R^{m}$)\footnote{More generally, it suffices to require that there exists an embedding of $M$ into $\R^d$ and a tubular neighbourhood of $M$ which admits a Lipschitz projection $\pi$ onto $M$. Indeed, with such assumption it is possible to recast the problem in the Euclidean setting, via the machinery introduced in \cite[\S 8]{Wh2} and the technique used in Lemma \ref{lemmaflat}.} and suppose that $\Sigma^n\subset M^m$ is a smooth, embedded, compact, oriented submanifold with (possibly empty) boundary which is a strictly stable critical point for a smooth, elliptic parametric functional $F$. Then there exist $\varepsilon>0$ and $C>0$ (depending on $\Sigma$ and  $M$) such that 
\begin{equation}\label{e:main_estimate}
F(S)\geq F(\Sigma)+C(\mathbb{F}(S-\Sigma))^2,
\end{equation}
whenever $0<\mathbb{F}(S-\Sigma)\leq\varepsilon$ and $S$ is an integral current on $M$, homologous to $\Sigma$.
\end{theorem}

Following \cite{Wh}, in Theorem \ref{th2} we exploit the previous result to prove a minimax property of unstable, but nondegenerate minimal submanifolds. 
%

\subsection{Idea of the proof of Theorem \ref{th1_w}}
Here is a sketch of the proof of Theorem \ref{th1_w}. For simplicity, we replace \eqref{e:main_estimate} with the weaker (non quantitative) inequality $F(S)>F(\Sigma)$, which would imply that $\Sigma$ is uniquely minimizing in the flat neighbourhood. The proof is by contradiction, and it is inspired by the technique used in \cite{CL}. We assume that for every $\delta>0$ we can select $S_\delta$, homologous to $\Sigma$, which satisfies $F(S_\delta)\leq F(\Sigma)$ and $0<\mathbb{F}(S_\delta-\Sigma)<\delta$. We denote $\eta_\delta:=\mathbb{F}(S_\delta-\Sigma)$ and define, for $\lambda> 0$, a penalized functional $F_{\delta,\lambda}$ as
$$F_{\delta,\lambda}(T):=F(T)+\lambda|\mathbb{F}(T-\Sigma)-\eta_\delta|\,.$$
We then consider integral currents 
$$R_{\delta,\lambda}\in{\rm{argmin}}\{F_{\delta,\lambda}(T):T\;{\rm{is\;homologous\;to\;}}\Sigma\}.$$
By definition, we have 
$$F(R_{\delta,\lambda})\leq F_{\delta,\lambda}(R_{\delta,\lambda})\leq F_{\delta,\lambda}(S_\delta)=F(S_{\delta})\leq F(\Sigma)\,,$$
which, in addition, implies $R_{\delta,\lambda}\neq \Sigma$, since equality would lead to the contradiction  
\[F_{\delta,\lambda}(R_{\delta,\lambda})= F(\Sigma)+\lambda\eta_{\delta} \geq F(S_{\delta})+\lambda\eta_\delta > F_{\delta,\lambda}(S_\delta) \geq F_{\delta,\lambda}(R_{\delta,\lambda})\,.\] Moreover, one can easily prove (Lemma \ref{lemmaconv}) that every $R_\lambda$ such that $\F(R_{\delta_i,\lambda}-R_\lambda) \to 0$ for some $\delta_i\searrow 0$ is a minimizer of 
$$F_{0,\lambda}(T):=F(T)+\lambda|\mathbb{F}(T-\Sigma)|.$$
We can also prove (Lemma \ref{lambdazero}) that for $\lambda$ large enough, the only minimizer of $F_{0,\lambda}$ is $\Sigma$ itself, hence, by standard compactness and lower semicontinuity properties, we can find a sequence $\delta_i\searrow 0$ such that the currents $R_{\delta_i,\lambda}$ converge to $\Sigma$. By the strict stability of $\Sigma$, the inequality $F(R_{\delta_i,\lambda})\leq F(\Sigma)$ would immediately imply the contradiction that $R_{\delta_i,\lambda}=\Sigma$, for every $i$ sufficiently large, if we could guarantee that $R_{\delta_i,\lambda}$ are globally parametrized as graphs of regular maps on the normal bundle of $\Sigma$, converging to 0 strongly, up to the boundary. On the other hand, this is the case because 
every $R_{\delta_i,\lambda}$ is an almost minimizer for $F$ (Lemma \ref{lemmalambdamin}) and therefore its ``graphicality'' and the strong convergence are ensured by the regularity theory for almost minimizers. 

\subsection{Comparison with results in the literature}
Federer proved in \cite{Fed2} a minimizing property for any extremal submanifold, among homologous surfaces which differ from it by a closed current of small mass. Extensions of White's result have been considered more recently by other authors in several contests. In \cite{MoR}, the authors prove that any smooth, oriented hypersurface of a Riemannian manifold of dimension $m\leq 7$, which has constant mean curvature and positive second variation with respect to variations fixing the volume, is uniquely area minimizing among homological oriented hypersurfaces in a small $L^1$-neighbourhood. Moreover the volume constraint can be dropped for minimal hypersurfaces. In \cite{AFM}, the authors prove a statment of similar nature for nonlocal isoperimetric problems, providing also quantitative estimates. Lastly, for a comparison on the quantitative part of Theorem \ref{th1_w}, we refer the reader to the paper \cite{DpM} where the authors prove that, for uniquely regular area minimizing hypersurfaces, the validity of quadratic stability inequalities is equivalent to the uniform positivity of the second variation of the area.

%
%
\section*{Acknowledgements} We would like to thank Emanuele Spadaro for several inspiring discussions. D.I. is supported by SNF grant 159403 \textit{Regularity questions in geometric measure theory}. A.M. is supported by the ERC grant 306247 \textit{Regularity of area minimizing currents}. Part of this work was conceived while A.M. was hosted by the Max Planck Institute in Leipzig.  He would like to warmly thank the institute for the support received.

\section{Notations and Preliminars}


\subsection{Integral Currents}
A \emph{rectifiable $n$-current} $T$ on a Riemannian manifold $M$ is a continuous linear functional on the space of smooth differential $n$-forms on $M$ admitting the following representation:
\begin{equation}\label{eq:rect_curr}
 \langle T;\omega\rangle:=\int_{E}\langle\omega(x),\tau(x)\rangle\theta(x)\,d\cH^n(x),\quad\forall\,\omega\in \mathscr{C}^\infty_c(M,\Lambda^{n}(TM)),
\end{equation}
where: 
\begin{itemize}
\item $E$ is a countably $n$-rectifiable set (see \S 11 of \cite{Sim83}) contained in $M$,
\item $\cH^n$ is the $n$-dimensional Hausdorff measure,
\item $\tau(x)\in\Lambda_n(T_xM)$ is the \emph{orientation} of $T$, i.e. a simple $n$-vectorfield with $|\tau(x)|=1$ and spanning the approximate tangent space $T_x E$, for 
$\cH^n$-a.e. $x\in E$,
\item $\theta$ is a function in $L^1_{loc}(\cH^{n}\res E)$ which is called the \emph{multiplicity} of $T$.
\end{itemize}
We define the mass $\mathbf{M}_A(T)$ of the 
$n$-current $T$ in the  Borel set $A\subset M$ as the measure $\theta\cH^{n}\res E(A)$. We will drop the subscript $A$ when $A=\R^m$. The Radon measure $\theta\cH^{n}\res E$ will be often denoted simply by $\|T\|$, while the vectorfield $\tau$ will be also denoted $\vec T$. Clearly to a compact, orientable, smooth, $n$-dimensional submanifold $\Sigma$ it is 
canonically associated a rectifiable $n$-current, whose mass coincides with the $n$-dimensional volume of $\Sigma$. The \emph{boundary} of a rectifiable $n$-current $T$ is the $(n-1)$-current 
$\partial T$ defined by the relation 
$$\langle\partial T,\phi\rangle=\langle T;d\phi\rangle\quad\forall\,\phi\in \mathscr{C}^\infty_c(M,\Lambda^{n-1}(TM)).$$

An \emph{integral $n$-current} $T$ is a rectifiable $n$-current with finite mass, such that the boundary $\partial T$ is also a rectifiable $(n-1)$-current with finite mass and the multiplicity both in $T$ and 
in $\partial T$ takes only integer values.\\

The space of currents is naturally endowed with a notion of weak$^*$- convergence. In some cases it is convenient to consider also the following notion of metric. The \emph{flat norm} of an integral $n$-current $T$ in the compact set $K$ is the quantity
\begin{equation}\label{e:def_flat}
\mathbb{F}_K(T)=\inf \{\mathbf{M}_K(S)+\mathbf{M}_K(R):T=S+\partial R,\;\;\; S,R\; {\rm{are\;integral\; currents}}\}.
\end{equation}
The flat norm $\F(T)$ of the current $T$ is obtained by removing the subscript $K$ in the previous formula.
\subsection{Parametric integrands}\label{s:Parametric integrands}
\noindent A \emph{parametric integrand} of degree $n$ on $\R^m$ is a continuous map
$$F:\R^m\times \Lambda^n(\R^m)\to\R$$
which takes non-negative values, is positively homogeneous in the second variable and satisfies
\begin{equation}\label{lambda}
\Lambda^{-1}\|\tau\|\leq F(x,\tau)\leq\Lambda\|\tau\|,
\end{equation}
for some $\Lambda>0$. The parametric integrand induces a functional (also denoted by $F$) on integral $n$-currents on $\R^m$ defined by
$$F(T):=\int_{\R^m}F(x,\vec{T}(x))d\|T\|(x).$$
For fixed $x\in \R^m$ we define $F_x$ to be the integrand obtained by ``freezing'' $F$ at $x$, i.e.
$$F_x(y,\tau):=F(x,\tau).$$
Let $F$ be a parametric integrand of degree $n$ on $\R^m$. We call $F$ \emph{elliptic} if there is $C>0$ such that for every $x\in \R^m$ it holds
$$F_x(T)-F_x(S)\geq C(\mathbf{M}(T)-\mathbf{M}(S)),$$
whenever $S$ and $T$ are compactly supported integral $n$-currents on $\R^m$, with $\partial S=\partial T$ and $S$ is represented by a measurable subset of an $n$-dimensional, affine subspace.\\
Let $F$ be an elliptic parametric functional on $\R^{m}$, $T$ and $S$ be two $n-$dimensional integral currents, and $A\subset \R^{m}$ a $\|T\|$-measurable set. We will use the following facts:
\begin{enumerate}
\item $F(T) = F(T\res A)+F(T\res A^{c})$,
\item $F(T+S) \leq F(T)+F(S)$,
\item $F$ is lower semicontinuous with respect to the flat convergence of integral currents (cf Theorem 5.1.5 in \cite{Fed}). 
\end{enumerate} 

\subsection{Stability}
Let $M$ be a Riemannian manifold of dimension $m$. Let $\Sigma\subset M$ be an $n$-dimensional, compact, smooth submanifold with (possibly empty) boundary, which is a stationary point for a parametric integrand $F$. 

Let $J$ be the Jacobi (or second variation) operator of $\Sigma$ for $F$, acting on the space of smooth sections of the normal bundle of $\Sigma$ which vanish on $\partial\Sigma$. 

We say that $\lambda$ is an eigenvalue of $J$ if there exists a non-trivial normal vector field $X$ which vanishes on $\partial\Sigma$ such that
$$JX-\lambda X=0.$$
The \emph{index} of $\Sigma$ is the (possibly infinite) number of negative eigenvalues of $J$ (counted with multiplicity). The \emph{nullity} of $\Sigma$ is the number of linearly independent normal vector fields $X$ vanishing on $\partial\Sigma$ and satisfying $JX=0$. We say that $\Sigma$ is \emph{non-degenerate} if $J$ has nullity zero. We say that $\Sigma$ is \emph{strictly stable} if all the eigenvalues of $J$ are bounded from below by a strictly positive constant.

\section{Proof of Theorem \ref{th1_w}: Minimality of $\Sigma$}
As discussed in the introduction of \cite{Wh}, by the results of \cite[\S 8]{Wh2}, it is sufficient to prove the theorem when $M=\R^m$. Throughout the paper we denote by $T$ an $n$-dimensional integral current on $\R^m$. \\
The purpose of this section is to prove Proposition \ref{th1_nonquant}; the following weaker (non quantitative) version of Theorem \ref{th1_w}. We deduce from it that $\Sigma$ is the unique minimizer in a small (flat) neighbourhood. 
The quantitative version will be proved in Section \ref{s:quant}, exploiting this result. 
\begin{proposition}\label{th1_nonquant}
Let $M^{m}$ be a smooth Riemannian manifold and suppose that $\Sigma^n\subset M^m$ is a smooth, embedded, compact, oriented submanifold with (possibly empty) boundary, which is  strictly stable for a smooth, elliptic parametric functional $F$. Then there exist $\varepsilon>0$ (depending on $\Sigma$ and $M$) such that 
\begin{equation}\label{e:main_estimate_nonquant}
F(S)> F(\Sigma),
\end{equation}
whenever $0<\mathbb{F}(S-\Sigma)\leq\varepsilon$ and $S$ is an integral current on $M$, homologous to $\Sigma$.
\end{proposition}

\subsection{Proof of Proposition \ref{th1_nonquant}: Penalized functionals}\label{s3.1} We start by assuming (by contradiction) that for any $\delta>0$ there exists an integral current $S_\delta \neq \Sigma$ with the properties 
\begin{flalign*}
\quad\text{(i) } &S_\delta \text{ is homologous to } \Sigma\,; &\\
\quad \text{(ii) } &\F(S_\delta - \Sigma) < \delta\,;&\\
\quad \text{(iii) }  &F(S_\delta) \leq F(\Sigma) \,.&
\end{flalign*}
 If we could guarantee that $S_\delta$ were, in addition, regular normal graphs over $\Sigma$, then the stability of $\Sigma$ would lead to the contradiction $\Sigma = S_\delta$ for small enough $\delta$. However, this is not the case in general. For this reason we fix a parameter $\lambda>0$ and replace each $S_\delta$ by a minimizer $R_{\delta,\lambda}$ of the penalized functional  
 \begin{equation}\label{d:penalizedfunctional}
 F_{\delta,\lambda}(T) := F(T) +\lambda|\F(T-\Sigma) -\eta_\delta|\,,
\end{equation}
 where for brevity we denoted $\eta_{\delta}:= \F(S_\delta-\Sigma)$. More precisely, we choose 
 \begin{equation}
R_{\delta,\lambda} \in \text{argmin}\{F_{\delta,\lambda}(T):T\;{\rm{is\;homologous\;to\;}}\Sigma\}\,,
\end{equation}
 which exists (although  it may well be not unique), because of the usual compactness theorem for uniformly mass bounded (see \eqref{lambda}) integral currents and the lower semicontinuity of the functional $F$ with respect to flat convergence.


Now if $\lambda$ is large enough we expect (ii) to hold also for $R_{\delta,\lambda}$, whereas (i) and (iii) follow from the definition: 
\[ F(R_{\delta,\lambda})\leq F_{\delta,\lambda}(R_{\delta,\lambda})\leq F_{\delta,\lambda}(S_\delta)=F(S_\delta)\leq F(\Sigma)\,.\]
The upshot is that the minimizers $R_{\delta,\lambda}$ are also "almost minimizers" for the functional $F$ (see Lemma \ref{lemmalambdamin}) and the desired graphicality is then a consequence of the regularity theory for almost minimizers.\\
Hence the task is to check that (ii) holds as well. More precisely, our aim is to find $\lambda>0$ and a sequence of $\delta_i\searrow 0$ such that $\F(R_{\delta_i,\lambda}-\Sigma)\to 0$. The first step in this direction consists in proving that every subsequential limit of $R_{\delta_i,\lambda}$ is a minimizer for the functional
$$F_{0,\lambda}(T):=F(T)+\lambda\mathbb{F}(T-\Sigma).$$

\begin{lemma}\label{lemmaconv}
Fix $\lambda>0$. Let $\delta_i \searrow 0$ and let $\F(R_{\delta_i,\lambda}-R_\lambda)\to 0$. Then we have
$$R_\lambda\in{\rm{argmin}}\{F_{0,\lambda}(T):T\;{\rm{is\;homologous\;to\;}}\Sigma\}.$$
\end{lemma}
\begin{proof}
 Since $\F(R_{\delta_i,\lambda} - R_\lambda)\to 0$ and $\eta_{\delta_i}\to 0$ we have, by lower semicontinuity
\begin{align*}
F_{0,\lambda}(R_\lambda)&=F(R_{\lambda})+\lambda\mathbb{F}(R_{\lambda}-\Sigma)\\
&\leq\liminf_{i\to\infty}\{F(R_{\delta_i,\lambda})+\lambda|\mathbb{F}(R_{\delta_i,\lambda}-\Sigma)-\eta_{\delta_i}|\}\\
&=\liminf_{i\to\infty} F_{\delta_i,\lambda}(R_{\delta_i,\lambda}).
\end{align*} 
 
Assume now by contradiction that there exist $\varepsilon>0$ and an $n$-dimensional integral current $S$, homologous to $\Sigma$, such that
$$F_{0,\lambda}(S)<F_{0,\lambda}(R_\lambda)-2\varepsilon.$$
Consider an integer $N$ such that $\lambda\eta_{\delta_N}\leq\varepsilon$. For every $M\geq N$ it holds
\begin{align*}
F_{\delta_M,\lambda}(S)&=F(S)+\lambda|\mathbb{F}(S-\Sigma)-\eta_{\delta_M}|\leq F(S)+\lambda\mathbb{F}(S-\Sigma)+\lambda\eta_{\delta_M}\\
&\leq F(S)+\lambda\mathbb{F}(S-\Sigma)+\varepsilon=F_{0,\lambda}(S)+\varepsilon<F_{0,\lambda}(R_\lambda)-\varepsilon\\
&\leq \liminf_{i\to\infty} F_{\delta_i,\lambda}(R_{\delta_i,\lambda})-\varepsilon,\leq \liminf_{i\to\infty} F_{\delta_i,\lambda}(S)-\varepsilon,
\end{align*}
which is a contradiction.
\end{proof}

The second step is to prove that, if $\lambda$ is sufficiently large, then the only minimizer of $F_{0,\lambda}$ is $\Sigma$. This is achieved in Lemma \ref{lambdazero}. To prove it we need the estimate of Lemma \ref{calibra}, which, on the other hand, is based on the following general fact. Roughly it states that, in small regimes, the flat norm of a closed current is realized by a minimal filling. For the sake of generality, only in the next lemma we consider the flat norm on the manifold $M$, i.e. we require that the currents $R$ and $S$ in \eqref{e:def_flat} are supported on $M$.

\begin{lemma}\label{lemmaflat}
Let $M^m$ be a smooth, compact manifold in $\R^d$ (or $M=\R^m$). Then there exists $\varepsilon_0=\varepsilon_0(M)>0$ such that for every $n$-dimensional integral current $T$ on $M$ with 
$\partial T=0$ and $\mathbb{F}(T)\leq\varepsilon_0$ there is an integral $(n+1)$-current $R$ on $M$ satisfying $\partial R=T$ and
$$\mathbf{M}(R)=\mathbb{F}(T).$$
\end{lemma}
\begin{proof}
We will assume that $M$ is a smooth, compact manifold in $\R^d$; the proof for $M=\R^m$ is identical. Fix $\delta>0$ and let $P$ and $Q$ be integral currents in $M$ satisfying $T=P+\partial Q$ and $\mathbf{M}(P)+\mathbf{M}(Q)\leq \mathbb{F}(T)+\delta$. Let $S$ be an integral $(n+1)$-current in $\R^d$ minimizing the mass among 
all integral currents with boundary equal to $P$ (notice that $\partial P=0$). For $\cH^{n+1}$-a.e. $x\in$ spt$(S)$, the monotonicity formula (see formula (17.3) of \cite{Sim83}) yields 
$$\mathbf{M}(S\res B_r(x))\geq\omega_{n+1}r^{n+1},$$ 
whenever $r<$ dist$(x,M)$. Moreover, the isoperimetric inequality (see Theorem 30.1 of \cite{Sim83}) gives 
$$\mathbf{M}(S)\leq C_1\mathbf{M}(P)^{\frac{n+1}{n}}.$$
This implies that, if $\mathbb{F}(T)$ is sufficiently small, the support of $S$ is contained in a small tubular neighbourhood of $M$. In particular, by the assumptions on $M$, the closest point projection $\pi$ on $M$ is (uniquely defined and) 
Lipschitz on this neighbourhood. By Lemma 26.25 of \cite{Sim83}, it follows that $\mathbf{M}(\pi_{\sharp}S)\leq C_2\mathbf{M}(S)$, therefore if $\mathbb{F}(T)$ is sufficiently small, then $\mathbf{M}(\pi_{\sharp}S)\leq\mathbf{M}(P)$.
Hence, setting $R_\delta=\pi_{\sharp}S+Q$, we have $T=\partial R_\delta$ and 
$$\mathbf{M}(R_\delta)\leq\mathbf{M}(\pi_{\sharp}S)+\mathbf{M}(Q)\leq\mathbb{F}(T)+\delta.$$
By the usual compactness and lower semicontinuity, we can take $R$ as a subsequential limit of any sequence $R_{\delta_i}$, for $\delta_i\to 0$. Clearly $R$ is supported on $M$.
\end{proof}

\begin{lemma}\label{calibra}
There exist $\varepsilon>0$ and a constant $C=C(\Sigma)$ such that, for every $n$-dimensional integral current $T$ on ${\R^m}$, homologous to $\Sigma$, such that $\mathbb{F}(\Sigma-T)\leq\varepsilon$, it holds
 $$F(\Sigma)-F(T)\leq C\mathbb{F}(\Sigma-T).$$
\end{lemma}
\begin{proof}
We present firstly a very simple proof of this fact, which is valid in the case $F(x,\tau)\equiv \|\tau\|$, i.e. $F(T)=\mathbf{M}(T)$. 

Let $\omega$ be a compactly supported $n$-form satisfying $\|\omega\|_\infty\leq 1$ and $\langle\omega;\tau_\Sigma\rangle=1$ 
whenever $\tau_\Sigma$ is a tangent unit vector orienting $\Sigma$ (which exists by smoothness of $\Sigma$). If $\varepsilon\leq\varepsilon_0$ in Lemma \ref{lemmaflat}, we can find an integral $(n+1)$-current $S$ satisfying
\begin{equation}\label{e_flat}
\mbox{$\partial S=\Sigma-T$\quad and \quad $\mathbf{M}(S)=\mathbb{F}(\Sigma-T)$.}
\end{equation}
Then we have
\begin{align*}
&F(\Sigma)-F(T)=\int_M 1 d\|\Sigma\|-\int_{\R^d} 1 d\|T\|\\
&\leq\int_{\R^d} \langle\omega (x);\tau_\Sigma(x)\rangle d\|\Sigma\|(x)-\int_{\R^d} \langle\omega (x);\vec{T}(x)\rangle d\|T\|(x)\\
&=\langle S;d\omega\rangle\leq\|d\omega\|_\infty\mathbf{M}(S)=\|d\omega\|_\infty\mathbb{F}(\Sigma-T).
\end{align*}

This completes the proof in the case $F(x,\tau)\equiv |\tau|$. In case $F$ is a convex functional, it is easy to adapt the previous argument.\\

In the general case, the only proof we are able to devise is more involved. In particular we need to exploit the stability of $\Sigma$ and we make use of Theorem 2 of \cite{Wh}. 

Let $d(x):=\text{dist}(x,\Sigma)$ and again choose $\varepsilon\leq \varepsilon_0$ in Lemma \ref{lemmaflat}. Let $S$ be an $(n+1)$-dimensional integral current such that $\partial S=T-\Sigma$ and $\mathbf{M}(S)=\mathbb{F}(T-\Sigma)$. Let $\varepsilon_1<\varepsilon_0$ be such that the open tubular neighbourhood of radius $\varepsilon_1$ centred at $\Sigma$, i.e. the set
$$B_{\varepsilon_1}(\Sigma):=\{x\in\R^m: \text{dist}(x,\Sigma)<\varepsilon_1\}\,,$$
is contained in the open tubular neighbourhood $U$ given by Theorem 2 of \cite{Wh}.\\
Denoting by $\langle S,d,t\rangle$ the ``slices'' of $S$ according to the function $d$, we have by standard properties of the slicing (see Lemma 28.5 (1) and (2) of \cite{Sim83}) that there exists $t\in(\varepsilon_1/2,\varepsilon_1)$ such that 
\begin{equation}\label{e:slicing}
\langle S, d, t\rangle=\partial(S\res B_t(\Sigma))-(\partial S)\res B_t(\Sigma) \quad\mbox{ and }\quad\mathbf{M}(\langle S, d, t\rangle)\leq \frac{2\mathbf{M}(S)}{\varepsilon_1}.
\end{equation}
Observe that $\tilde T:=T\res B_t(\Sigma) +\langle S, d, t\rangle$ is supported in $B_{\varepsilon_1}(\Sigma)$ and it is homologous to $\Sigma$, indeed 

\begin{align*}
\tilde T-\Sigma &=T\res B_t(\Sigma) +\partial(S\res B_t(\Sigma))-(\partial S)\res B_t(\Sigma)-\Sigma\\
&=T\res B_t(\Sigma) +\partial(S\res B_t(\Sigma))-(T-\Sigma)\res B_t(\Sigma)-\Sigma\res B_t(\Sigma)\\
&=\partial(S\res B_t(\Sigma)).
\end{align*}

Eventually we compute:
\begin{align*}
F(\Sigma)-F(T)&=F(\Sigma)-(F(T\res B_t(\Sigma))+F(\langle S, d, t\rangle))+F(\langle S, d, t\rangle)-F(T\res B_t(\Sigma)^c)\\
&\leq F(\Sigma)-F(\tilde{T})+F(\langle S, d, t\rangle),
\end{align*}
where we used the fact that parametric integrands are additive on currents supported on disjoint sets and subadditive for general rectifiable currents, and $F\geq0$.
Now, by Theorem 2 of \cite{Wh}, $F(\Sigma)-F(\tilde{T})<0$.
So we can conclude that, whenever $T$ is homologous to $\Sigma$ and $F(T-\Sigma)\leq\varepsilon_1$, it holds:
$$F(\Sigma)-F(T)\leq F(\langle S, d, t\rangle)\stackrel{\eqref{lambda}}{\leq}\Lambda\mathbf{M}(\langle S, d, t\rangle)\stackrel{\eqref{e:slicing}}{\leq}\frac{2\Lambda\mathbf{M}(S)}{\varepsilon_1}\stackrel{\eqref{e_flat}}{=}\frac{2\Lambda}{\varepsilon_1}\mathbb{F}(T-\Sigma).\qedhere$$
\end{proof}


\begin{lemma}\label{lambdazero}
 There exists $\lambda_0>0$ such that, for every $\lambda>\lambda_0$, there holds
$${\rm{argmin}}\{F_{0,\lambda}(T):T\;{\rm{is\;homologous\;to\;}}\Sigma\}=\{\Sigma\}.$$
\end{lemma}
\begin{proof}
 Assume by contradiction there exist $\lambda_i\to\infty$ and $S_i\neq\Sigma$ such that $S_i\;{\rm{is\;homologous\;to\;}}\Sigma$ and 
$$S_i\in{\rm{argmin}}\{F_{0,\lambda_i}(T):T\;{\rm{is\;homologous\;to\;}}\Sigma\}.$$
Notice that $F_{0,\lambda_i}(\Sigma)=F(\Sigma)$ for every $i$, therefore we have:
$$F(S_i)+\lambda_i\mathbb{F}(S_i-\Sigma)=F_{0,\lambda_i}(S_i)\leq F_{0,\lambda_i}(\Sigma)=F(\Sigma).$$
Hence
$$\lambda_i\mathbb{F}(S_i-\Sigma)\leq F(\Sigma)-F(S_i),$$
which, for $i$ sufficiently large, contradicts Lemma \ref{calibra}.
\end{proof}

\subsection{Proof of Proposition \ref{th1_nonquant}: Almost minimizers}

Now, fixing $\lambda>\lambda_0$, we have a sequence of integral currents $R_{\delta_i,\lambda}$ homologous to $\Sigma$ such that $\F(R_{\delta_i,\lambda}-\Sigma)\to 0$ and $F(R_{\delta_i,\lambda})\leq F(\Sigma)$, i.e. we managed to replace the original sequence $S_{\delta_i}$ (a sequence of minimizers of a functional under an additional constraint) by  a sequence of global minimizers of a penalized functional. As already mentioned, being minimizers of the penalized functionals guarantees almost minimality properties, and by the regularity theory for such almost minimizers, one can deduce that the $R_{\delta_i,\lambda}$ are regular, normal graphs over $\Sigma$, for sufficiently large $i$. From this information, one would be able to conlcude the proof as in \cite{Wh}. Actually, we do not need to repeat the final part of that proof, but we can recast the problem in White's setting and exploit his result, once we prove that the almost minimality of the $R_{\delta_i,\lambda}$'s implies that, for $i$ sufficiently large, they are supported in a small tubular neighbourhood of $\Sigma$ (cf. Lemma \ref{l:hausdorffconvergence}).\\

\noindent We will adopt a special case of the notion of almost minimality introduced in \cite{DS}. Given an elliptic parametric functional $F$ and $C>0$, we say that an $n$-dimensional integral current $S$ is $C$-\emph{almost minimizing} for $F$, if
\begin{equation}\label{e:almostminimizing}
F(S)\leq F(S+X)+Cr\M(S\res K+X) \,,
\end{equation}
whenever $X$ is a closed $n$-current with support in a compact set $K$, which is contained in a ball $B_r= B_r(x_0) \subset \R^{m}$ of radius $r$ and center $x_0\in\R^{m}$. For the regularity theory it is sufficient to have \eqref{e:almostminimizing} only for small radii $r$ and for all $X$ satisfying in addition $\F(X)<1$. \\
As we show in the next lemma, the currents $R_{\delta,\lambda}$ are $C$-almost minimizing. This follows almost directly from the following stronger property of $R_{\delta,\lambda}$. Fix any closed $n$-dimensional current $X$. Then $R_{\delta,\lambda}+X$ is a competitor in the optimization of $F_{\delta,\lambda}$ and consequently  
 \begin{equation*}
F(R_{\delta,\lambda})+\lambda|\F(R_{\delta,\lambda}-\Sigma) -\eta_\delta| \leq F(R_{\delta,\lambda}+X) +\lambda|\F(R_{\delta,\lambda}+X-\Sigma)-\eta_\delta| \,.
\end{equation*}
This implies that 
\begin{equation}\label{e:almostmin4}
F(R_{\delta,\lambda})\leq F(R_{\delta,\lambda}+X)+\lambda\F(X)\,.
\end{equation}

\begin{lemma}\label{lemmalambdamin}
For every $\delta,\lambda>0,$ $R_{\delta,\lambda}$ are $C$-almost minimizing for $C = \frac{4\Lambda^{2}\lambda}{n+1}$.
\end{lemma}
\begin{proof}
The first observation we make is that, due to the minimality of $R_{\delta,\lambda}$ with respect to $F_{\delta,\lambda}$, it's not possible write $R_{\delta,\lambda}$ inside a sufficiently small ball as the sum of a closed current and another current which has much smaller mass. To make this precise fix $r<r_0:=\frac{n+1}{4\lambda\Lambda}$. We claim that there exist no closed currents $X$ supported in some compact $K\subset B_r\subset \R^{m}$ with $\F(X)<1$ such that 
\begin{equation}\label{e:almostmin1}
4\Lambda^{2}\M(R_{\delta,\lambda}\res K + X) \leq \M( X) \,.
\end{equation}
Indeed, assume by contradiction that we can find a current $X$ and set $K$ with such property. By \eqref{e:almostmin4} we have 
\[ F(R_{\delta,\lambda})\leq F(R_{\delta,\lambda}+X)+\lambda\F(X)\,.\]
Since $F$ is additive on currents with disjoint supports and $X$ is supported in $K$, we can subtract $F(R_{\delta,\lambda}\res K^{c})$ from both sides to get 
\[ F(R_{\delta,\lambda}\res K ) \leq F(R_{\delta,\lambda}\res K+X)+\lambda\F(X)\,.\]
We estimate $\F(X)$ by the mass of the cone over $X$ to get 
\begin{align}
F(R_{\delta,\lambda}\res K)&\leq F(R_{\delta,\lambda}\res K +X) +\lambda \frac{r}{n+1}\M(X) \nonumber\\
&\leq \Lambda \M(R_{\delta,\lambda}\res K+X)+\lambda\frac{r}{n+1}\M(X)\nonumber\\
&< \frac{1}{2\Lambda}\M(X)\,, \label{e:almostmin3}
\end{align}
by \eqref{e:almostmin1} and the assumption on $r$. Observe that 
\[ 4\Lambda^{2}\M(R_{\delta,\lambda}\res K+X)\leq \M(X)\leq \M(X+R_{\delta,\lambda}\res K)+\M(R_{\delta,\lambda}\res K)\,,\]
and so 
\[ (4\Lambda^{2}-1)\M(R_{\delta,\lambda}\res K+X)\leq \M(R_{\delta,\lambda}\res K)\,.\]
Since we assume without loss of generality that $\Lambda>1$, this implies
\[ \M(X ) \leq \M(R_{\delta,\lambda}\res K+X)+ \M(R_{\delta,\lambda}\res K) \leq 2\M(R_{\delta,\lambda}\res K )\leq 2\Lambda F(R_{\delta,\lambda}\res K) \,.\]
Plugging into \eqref{e:almostmin3} yields the contradiction 
\[ F(R_{\delta,\lambda} \res K ) < F(R_{\delta,\lambda}\res K) \,.\]
Consequently, when $r$ is small enough, any integral current $X$ as in the definition satisfies additionally 
\[ \M(X) \leq 4\Lambda^{2}\M(R_{\delta,\lambda}\res K+X) \,.\]
Combining this with \eqref{e:almostmin4} we immediately get 
\begin{align*}
 F(R_{\delta,\lambda})&\leq F(R_{\delta,\lambda}+X) +\lambda\F(X) \leq F(R_{\delta,\lambda}+X)+\lambda \frac{r}{n+1}\M(X) \\&\leq F(R_{\delta,\lambda}+X)+\frac{4\Lambda^{2}\lambda}{n+1}r\M(R_{\delta,\lambda}\res K+X)\,,
 \end{align*} 
proving $C$-almost minimality with $C = \frac{4\Lambda^{2}\lambda}{n+1}$. 
%
\end{proof}

\subsection{Proof of Proposition \ref{th1_nonquant}: Hausdorff convergence and conclusion}
\begin{lemma} \label{l:hausdorffconvergence} Assume $R_i$ is a sequence of integral currents which are $C-$almost minimizing for an elliptic parametric functional $F$ and which converge flat to $\Sigma$. Then 
\[\lim_{i\to\infty}\sup_{x\in spt(R_i)}\{\mbox{\rm dist}(x,\Sigma)\}= 0.\]
\end{lemma}

\begin{proof}
Assume by contradiction, up to passing to a suitable subsequence, that there exists $r>0$ such that for every $i\in\N$ there exists $x_i$ in the support of $R_i$ satisfying $$\text{dist}(x_i, \Sigma)>2r.$$
Without loss of generality we can assume $2r<r_0$. By the density lower bound (see Lemma 2.1 in \cite{DS}) there exists $D>0$ such that
\begin{equation}\label{haus2}
F(R_i\res B_r(x_i))\geq D r^n, 
\end{equation}
for every $i\in\N$. This contradicts the almost minimality of $R_i$ for $i$ large enough.
Indeed, since $\F(R_i-\Sigma)\to 0$, by Lemma \ref{lemmaflat} we can find integral $(n+1)-$currents $S_i$ satisfying 
\[ \partial S_i =R_i-\Sigma\, \quad \text{ and }\quad  \M(S_i) =\F(R_i-\Sigma) \to 0\,.\]
We consider the slices $\langle S_i,d_i,\rho_i\rangle$ of $S_i$ with respect to the distance function $d_i(x):=|x-x_i|$ for some $\rho_i\in (r,2r) $ satisfying $R_i\res \partial B_{\rho_i}(x_i) = 0$ and
\[ \langle S_i,d_i,t_i\rangle = \partial(S_i\res B_{\rho_i}(x_i))-\partial S_i\res B_{\rho_i}(x_i) \quad \text{ and }\quad \M(\langle S_i,d_i,t_i\rangle) \leq \frac{\M(S_i)}{r}\,.\]
Set $X_i = -\partial(S_i\res B_{\rho_i}(x_i))$. Then $X_i$ is supported in $\overline{B_{\rho_i}}(x_i)\subset B_{2r}(x_i)$ and $\F(X_i)\leq \M(S_i) <1$ if $i$ is large enough. Hence, by the almost minimality, it holds 
\[ F(R_i)\leq F(R_i+X_i) +C(2r)\M(R_i\res \overline{B_{\rho_i}}(x_i)+X_i)\,.\]
 We observe
\begin{align*}
R_i+X_i &= \partial S_i +\Sigma -\partial (S_i\res B_{\rho_i}(x_i))=\partial S_i\res (B_{\rho_i}(x_i))^{c}+\Sigma - \langle S_i,d_i,\rho_i\rangle \\
&=R_i\res(B_{\rho_i}(x_i))^{c}-\langle S_i,d_i,\rho_i\rangle\,,
\end{align*}
and 
\begin{align*} R_i\res \overline{B_{\rho_i}}(x_i) +X_i &= R_i+X_i - R_i\res (\overline{B_{\rho_i}}(x_i))^{c} 
= R_i\res(B_{\rho_i}(x_i))^{c}-R_i\res (\overline{B_{\rho_i}}(x_i))^{c}-\langle S_i,d_i,\rho_i\rangle\\
&= R_i\res \partial B_{\rho_i}(x_i) - \langle S_i,d_i,\rho_i\rangle =- \langle S_i,d_i,\rho_i\rangle\,,
\end{align*}
since $R_i\res \partial B_{\rho_i}(x_i) = 0$. Consequently, we find
\begin{align*}
F(R_i) &\leq F(R_i\res(B_{\rho_i}(x_i))^{c}) +F( \langle S_i,d_i,\rho_i\rangle)\\
&\leq F(R_i)-F(R_i\res B_{\rho_i}(x_i)) +(\Lambda+2Cr)\M(\langle S_i,d_i,\rho_i\rangle)
\end{align*}
Observe that $B_r(x_i)\subset B_{\rho_i}(x_i)$ and hence with \eqref{haus2} we infer 
\begin{align*} F(R_i\res B_{\rho_i}(x_i)) &\geq \Lambda^{-1}\M(R_i\res B_{\rho_i}(x_i)) \geq \Lambda^{-1}\M(R_i\res B_{r}(x_i))\geq \Lambda^{-2}F(R_i\res B_r(x_i)) \\&\geq \Lambda^{-2}Dr^{n}\,.\end{align*}
If $I$ is so large that for any $i\geq I$ we have 
\[\M(\langle S_i,d_i,\rho_i\rangle) < \frac{1}{2}\Lambda^{-2}(\Lambda +2Cr)^{-1}Dr^{n}\,,\]
then we find the contradiction 
\[F(R_i) \leq F(R_i)-\Lambda^{-2}Dr^{n}+(\Lambda+2Cr)\M(\langle S_i,d_i,\rho_i\rangle) < F(R_i)\,.\qedhere\]
\end{proof}

Consequently, applying Lemma \ref{l:hausdorffconvergence} to the sequence  $(R_{\delta_i, \lambda})_{i\geq 1}$ for some $\lambda> \lambda_0$ and $\delta_i\searrow 0$, we get  a sequence  of integral currents which are supported in arbitrarily small tubular neighbourhoods of $\Sigma$ and satisfy 
$$F(R_{\delta_i, \lambda})\leq F(\Sigma),$$
and this implies, by Theorem 2 of \cite{Wh}, that $R_{\delta_i, \lambda}=\Sigma$, for $i$ sufficiently large, which is the final contradiction.\\

\section{Proof of Theorem \ref{th1_w}: Quantitative Estimate}\label{s:quant}
To prove the quantitative estimate \eqref{e:main_estimate} we will redo most of the proof in the previous section but slightly change the penalized functional $F_{\delta,\lambda}$. We again argue by contradiction and assume that for every $\delta>0$ there exists $\tilde S_\delta\neq\Sigma$ which is homologous to $\Sigma$, has flat distance $\tilde\eta_\delta:= \F(\tilde S_\delta- \Sigma) < \delta$ from $\Sigma$  and satisfies 
\begin{equation}\label{defstildadelta}
F(\tilde S_\delta)\leq F(\Sigma)+C(\F(\tilde S_\delta-\Sigma))^{2}\,.
\end{equation}

The constant $C$ has to be thought as fixed for the moment and it will be chosen only at the end of the proof. Define
the penalized functional $\tilde F_{\delta,\lambda}$ as 
\begin{equation}\label{defSdelta}
\tilde F_{\delta,\lambda} (T) := F(T) +\lambda (\F(T-\Sigma)-\tilde \eta_\delta)^{2}\,.
\end{equation}
The semicontinuity of $\tilde F_{\delta,\lambda}$ follows immediately from the semicontinuity of $F$. Therefore we can again consider minimizers $\tilde R_{\delta,\lambda}$ of the penalized functional $\tilde F_{\delta,\lambda}$ among all integral currents homologous to $\Sigma$. By repeating the steps of the previous section we can find a nice subsequence converging to $\Sigma$.
\begin{lemma}\label{l:rightsequence} 
 There exists $\lambda_0 >0$ such that for any $\lambda\geq\lambda_0$ there is a sequence $\delta_i\searrow 0$ such that $\F(\tilde R_{\delta_i,\lambda}-\Sigma)\to 0 $ for $i \to +\infty$.
\end{lemma} 
\begin{proof} 
The argument is again divided into two parts: firstly, any integral current $\tilde R_\lambda $ which is a subsequential limit of $\tilde  R_{\delta,\lambda}$ is a minimizer of 
\[ \tilde F_{0,\lambda} (T) := F(T)+\lambda \F(T-\Sigma)^{2}\,.\]
This part is entirely similar to the one in Lemma \ref{lemmaconv}: let $\tilde R_{\lambda}, \tilde R_{\delta_i,\lambda} $ be such that $\F(\tilde R_{\delta_i,\lambda}-\tilde R_\lambda ) \to 0$ and assume by contradiction that there exists $\varepsilon>0$ and a $n$-dimensional integral current $S$ homologous to $\Sigma$ such that
\begin{equation}\label{defS}
\tilde F_{0,\lambda}(S)<\tilde F_{0,\lambda}(\tilde R_\lambda)-2\varepsilon.
\end{equation}
Consider an integer $N$ such that $\lambda\tilde \eta_{\delta_M}^2\leq\varepsilon$ for every $M\geq N$. Since by the semicontinuity of $\tilde F_{\delta,\lambda}$ we have 
\begin{align*}
\tilde F_{0,\lambda}(\tilde R_\lambda)&= F(\tilde R_{\lambda})+\lambda\mathbb{F}(\tilde R_{\lambda}-\Sigma)^{2}\\
&\leq\liminf_{i\to\infty}\{F(\tilde R_{\delta_i,\lambda})+\lambda(\mathbb{F}(\tilde R_{\delta_i,\lambda}-\Sigma)-\tilde \eta_{\delta_i})^2\}\\
&=\liminf_{i\to\infty} \tilde F_{\delta_i,\lambda}(\tilde R_{\delta_i,\lambda})\,,
\end{align*} 
we get for arbitrary $M\geq N$ the contradiction 
\begin{align*}
\tilde F_{\delta_M,\lambda}(S)&= F(S)+\lambda(\mathbb{F}(S-\Sigma)-\tilde \eta_{\delta_M})^2\\
&=  F(S)+\lambda(\mathbb{F}(S-\Sigma))^2+\lambda\tilde \eta_{\delta_M}^2-2\lambda\F(S-\Sigma)\tilde\eta_{\delta_M}\\
&\leq  F(S)+\lambda(\mathbb{F}(S-\Sigma))^2+\varepsilon=\tilde F_{0,\lambda}(S)+\varepsilon\stackrel{\eqref{defS}}{<}\tilde F_{0,\lambda}(\tilde R_\lambda)-\varepsilon\\
&\leq \liminf_{i\to\infty} \tilde F_{\delta_i,\lambda}(\tilde R_{\delta_i,\lambda})-\varepsilon,\leq \liminf_{i\to\infty} \tilde F_{\delta_i,\lambda}(S)-\varepsilon\,.
\end{align*}
Secondly, we claim that if $\lambda\geq \lambda_0 $ is large enough then the only minimizer of $ \tilde F_{0,\lambda}$ is $\Sigma$ itself. Consider $\lambda>\frac{F(\Sigma)}{\varepsilon^2}$, where $\varepsilon$ is the value defined in Proposition \ref{th1_nonquant} and assume by contradiction that there exists $S_\lambda\neq\Sigma$ such that $\tilde F_{0,
\lambda}(S_\lambda)\leq \tilde F_{0,\lambda}(\Sigma)$. Then we have
\begin{equation}\label{e:rudolph}
\tilde F_{0,\lambda}(S_\lambda)= F(S_\lambda)+\lambda (\F(S_\lambda-\Sigma))^2> F(\Sigma)\frac{(\F(S_\lambda-\Sigma))^2}{\varepsilon^2}.
\end{equation}
From the last inequality, since $\tilde F_{0,\lambda}(S_\lambda)\leq \tilde F_{0,\lambda}(\Sigma)=F(\Sigma)$, it follows that $\F(S_\lambda-\Sigma)<\varepsilon$. On the other hand, if $\F(S_\lambda-\Sigma)<\varepsilon$, by Proposition \ref{th1_nonquant}, it follows that $F(\Sigma)<F(S_\lambda)$, hence
$$\tilde F_{0,\lambda}(\Sigma)=F(\Sigma)< F(S_\lambda)< F(S_\lambda)+\lambda(\F(S_\lambda-\Sigma))^2=\tilde F_{0,\lambda}(S_\lambda),$$
which is a contradiction.
\end{proof}

At this point, for every $\lambda>\lambda_0$ we have a sequence $(\tilde R_{\delta_i,\lambda})$ of integral currents which are homologous to $\Sigma$ and converge flat to it. As before, the  $(\tilde R_{\delta_i,\lambda})$ are almost minimizers if $i$ is large enough: this follows as in Lemma \ref{lemmalambdamin}, once we observe that for any closed $n$-current $X$ with support in a compact $K\subset B_r(x_0)$ and with  bound $\F(X)<1$ we have the following inequality 
\begin{equation}\label{e:amproperty}
F( \tilde R_{\delta_i,\lambda}) \leq F(\tilde R_{\delta_i,\lambda}+X) +2\lambda\F(X) \,. 
\end{equation}
The latter inequality on the other hand is a consequence of 
\begin{align*} F(\tilde R_{\delta_i,\lambda}) +\lambda(\F(\tilde R_{\delta_i,\lambda}-\Sigma)-\tilde\eta_{\delta_i})^{2}&=\tilde F_{\delta_i,\lambda}(\tilde R_{\delta_i,\lambda})
\leq \tilde F_{\delta_i,\lambda}(\tilde R_{\delta_i,\lambda}+X)\\&= F(\tilde R_{\delta_i,\lambda}+X) +\lambda(\F(\tilde R_{\delta_i,\lambda}+X-\Sigma)-\tilde \eta_{\delta_i})^{2}\,,\end{align*}
since this implies
\begin{align*} 
F( \tilde R_{\delta_i,\lambda} ) \leq & F(\tilde R_{\delta_i,\lambda}+X) 
        -2\lambda \tilde \eta_{\delta_i}\left (\F(\tilde R_{\delta_i,\lambda_0}+X-\Sigma)-\F(\tilde R_{\delta_i,\lambda_0}-\Sigma)\right ) \\         
       & +\lambda \left (\F(\tilde R_{\delta_i,\lambda}+X-\Sigma)+\F(\tilde R_{\delta_i,\lambda}-\Sigma)\right ) \left(\F(\tilde R_{\delta_i,\lambda_0}+X-\Sigma)-\F(\tilde R_{\delta_i,\lambda_0}-\Sigma)\right )\\
     \leq& F(\tilde R_{\delta_i,\lambda}+X) -2\lambda \tilde \eta_{\delta_i}\left (\F(\tilde R_{\delta_i,\lambda_0}+X-\Sigma)-\F(\tilde R_{\delta_i,\lambda_0}-\Sigma)\right )\\ &+\lambda \F(X) \left (\F(\tilde R_{\delta_i,\lambda}+X-\Sigma)+\F(\tilde R_{\delta_i,\lambda}-\Sigma)\right) \\     
\leq & F( \tilde R_{\delta_i,\lambda_0}+X)+\lambda\F(X)\left ( \F(X) +2\tilde \eta_{\delta_i} +2\F(\tilde R_{\delta_i,\lambda}-\Sigma)\right )\,,\end{align*}
which yields \eqref{e:amproperty} for $i$ large enough.
The Hausdorff convergence of $R_{\delta_i,\lambda}$ to $\Sigma$ then follows again from Lemma \ref{l:hausdorffconvergence}.\\
The almost minimality of the members of the sequence and consequent Hausdorff convergence together with the stability of $\Sigma$ now imply that $\tilde R_{\delta_i,\lambda} = \Sigma$ for $i$ large enough. In particular we will use the well known fact that, for $i$ sufficiently large, we can write $\tilde R_{\delta_i,\lambda}$ as a graph of normal vectorfield $u_i$ over $\Sigma$, converging to $0$ in $C^{1,\beta}$, up to the boundary (see \cite{DS}). In turn, this fact implies that, for $i$ sufficiently large, it holds 
$$\|u_i\|^2_{W^{1,2}}\geq c_0\|u_i\|^2_{L^1}\geq c_1\left (\F(\tilde R_{\delta_i,\lambda}-\Sigma)\right )^2,$$
for some $c_1>0$ (depending on $\Sigma$). Indeed the first inequality is just H\"older's inequality and the second is due to the bound on the $C^1$-norm of $u_i$. Now we compute
\[ \tilde F_{\delta_i,\lambda}(\tilde R_{\delta_i,\lambda}) \leq \tilde F_{\delta_i,\lambda}(\tilde S_{\delta_i}) \stackrel{\eqref{defSdelta}}{=}F(\tilde S_{\delta_i})\stackrel{\eqref{defstildadelta}}{\leq} F(\Sigma)+C(\F(\tilde S_{\delta_i}-\Sigma))^{2}\,,\]
and by the stability of $\Sigma$, denoting $4\eta_0$ the (strictly positive) minimal eigenvalue of the Jacobi operator of $F$ for $\Sigma$, we get, for every $\lambda>\lambda_0$ and for $i$ sufficiently large, 
\[ F(\tilde R_{\delta_i,\lambda})\geq F(\Sigma) +\eta_0\|u_i\|_{W^{1,2}}^{2} \geq F(\Sigma)+\eta_0c_1(\F(\tilde R_{\delta_i,\lambda}-\Sigma))^{2}\,.\] 
We deduce that, for $i$ sufficiently large,
 \begin{align*} F(\Sigma)+\eta_0c_1(\F(\tilde R_{\delta_i,\lambda}-\Sigma))^{2} +\lambda (\F(\tilde R_{\delta_i,\lambda}-\Sigma)-\tilde \eta_{\delta_i})^{2}&\leq F(\tilde R_{\delta_i,\lambda})+\lambda (\F(\tilde R_{\delta_i,\lambda}-\Sigma)-\tilde \eta_{\delta_i})^{2} \\&=\tilde F_{\delta_i,\lambda}(\tilde R_{\delta_i,\lambda}) \leq F(\Sigma)+C\tilde \eta_{\delta_i}^{2}\,.\end{align*}
 Since $\tilde \eta_{\delta_i}^{2}\leq 2(\tilde \eta_{\delta_i}-\F(\tilde R_{\delta_i,\lambda}-\Sigma))^{2}+2(\F(\tilde R_{\delta_i,\lambda}-\Sigma))^{2}$ we infer 
 \[ \eta_0c_1(\F(\tilde R_{\delta_i,\lambda}-\Sigma))^{2} +\lambda (\F(\tilde R_{\delta_i,\lambda}-\Sigma)-\tilde \eta_{\delta_i})^{2}\leq 2C((\tilde \eta_{\delta_i}-\F(\tilde R_{\delta_i,\lambda}-\Sigma))^{2}+(\F(\tilde R_{\delta_i,\lambda}-\Sigma))^{2})\,,\]
or equivalently 
\[ (\eta_0c_1-2C)(\F(\tilde R_{\delta_i,\lambda}-\Sigma))^{2}\leq (2C-\lambda)(\tilde \eta_{\delta_i}-\F(\tilde R_{\delta_i,\lambda}-\Sigma))^{2}\,,\]
which is not possible for $C=C(\eta_0)$ small enough and $\lambda$ large enough. 

\subsection{Final remarks}
In \cite{Wh}, the author proves that when $\Sigma$ is an unstable but non-degenerate critical point, then it is the unique solution to a suitable minimax problem in a certain tubular neighbourhood. The corresponding statement in our setting is the following.

\begin{theorem}\label{th2}
Let $\Sigma^n\subset M^m$ be a smooth, embedded, compact submanifold with (possibly empty) boundary. Suppose also that $\Sigma$ is a nondegenerate critical point with index $k>0$ for a smooth parametric elliptic functional $F$. Then there exist $\varepsilon>0$, $\delta>0$, $C>0$ and a 
smooth $k$-parameter family $(\Sigma_v)_{v\in \overline{B}^k_\varepsilon}$ of embedded surfaces, each homologous to $\Sigma=:\Sigma_0$, with the following properties:
\begin{enumerate}
 \item for every $v\in \overline{B}^k_\varepsilon$ it holds $\F(\Sigma_v-\Sigma)<\delta$ and
$$F(\Sigma_v)-F(\Sigma_0)\leq-\varepsilon|v|^2;$$
 \item if $(\widetilde{\Sigma}_v)_{v\in \overline{B}^k_\varepsilon}$ is any other $k$-parameter family of integral currents, each homologous to $\Sigma$, which is continuous (with respect to the weak$^*$ convergence of currents) 
and satisfies 
$$\F(\tilde\Sigma_v-\Sigma)<\delta\quad\mbox{ and }\quad\widetilde{\Sigma}_v=\Sigma_v, \mbox{ for } v\in\partial B^k_\varepsilon,$$ 
then
\begin{equation}\label{eq:conj}
\sup_v \{F(\widetilde{\Sigma}_v)-C(\mathbb{F}(\widetilde{\Sigma}_v-\Sigma))^2\}\geq F(\Sigma).
\end{equation}
\end{enumerate}
\end{theorem}
For a non quantitative version of this result, one can replace \eqref{eq:conj} simply with the inequality  
$\sup_v \{F(\widetilde{\Sigma}_v)\}\geq F(\Sigma)$,
with the additional information that the inequality is strict, unless $\tilde\Sigma_v=\Sigma$ for some $v$.
Following \cite{Wh}, to prove  Theorem \ref{th2}, it suffices to prove Theorem \ref{th1_w} with $F$ replaced by a functional $G$ of the form 
$$G(T):=F(T)+\phi\left(\int_{\R^d} fF\,d\|T\|\right),$$
where $f:\R^d\to\R^k$ is a continuous map and $\phi:\R^k\to\R$ is a $C^2$ function with  $f\equiv 0$ on $\Sigma$, $\phi(0)=0$ and
\begin{equation}\label{e:bounds}
\sup|D\phi|\cdot|f|<1\,.
\end{equation}

Most of our strategy goes through almost verbatim. Indeed, one can easily prove the analogue of Lemma \ref{lemmaconv}, \ref{calibra} and Lemma \ref{lambdazero}. This is achieved using the simple but crucial observation, already contained in \cite{Wh}, that if $T$ and $S$ are $n-$dimensional integral currents on $\R^{d}$ then there exists $L\in \R^{k}$ (depending on $T$ and $S$) satisfying $|L|\leq \sup|D\phi|$ and 
\begin{equation}\label{e:philinear}
\phi\left (\int_{\R^{d}} fFd\|T\| \right ) - \phi\left (\int_{\R^{d}}fF d\|S\|\right ) = \int_{\R^{d}} \langle L, f\rangle F\,d\left (\|T\|-\|S\|\right )\,.
\end{equation} 
In particular, by \eqref{e:bounds} this implies the lower-semicontinuity of $G$ and the fact that $G$ is bounded from above and below by a multiple of the mass.\\
It turns out that the minimizers $R_{\delta,\lambda}$ of the corresponding functionals $G_{\delta,\lambda}$ are again $C-$almost minimizing for the functional $F$.
Since the almost minimality is not completely obvious, we sketch a proof here. For a fixed closed current $X$ which is supported in compact set $K\subset B_r(x)$ we get from $G_{\delta,\lambda}(R_{\delta,\lambda})\leq G_{\delta,\lambda}(R_{\delta,\lambda}+X)$ and \eqref{e:philinear} that 
\[ \tilde F(R_{\delta,\lambda}) \leq \tilde F(R_{\delta,\lambda}+X) +\lambda\F(X)\,,\]
where $\tilde F(T) = \int(1+\langle L, f\rangle) F\,d\|T\|$ and $L\in\R^{k}$ with $|L|\leq \sup|D\phi|$ depends on $X$ and $R_{\delta,\lambda}$. Arguing as in Lemma \ref{lemmalambdamin} leads to the inequality 
\begin{equation}\label{e:fakeam}
\tilde F(R_{\delta,\lambda}) \leq \tilde F(R_{\delta,\lambda}+X) +Cr\M(R_{\delta,\lambda}\res K+X)\,,
\end{equation} 
from which one would be tempted to conclude the almost minimality of $R_{\delta,\lambda}$ for the elliptic, parametric integrand $\tilde F$. However, the functional $\tilde F$ depends on the fixed $X$. \\
Nevertheless, since the bounds on $\tilde F$ \emph{don't} depend on $X$ or $R_{\delta,\lambda}$,  inequality \eqref{e:fakeam} yields the following estimate for the frozen integrands: 
\[ \tilde F_x(R_{\delta,\lambda}\res K) \leq (1+Cr)\tilde F_x(R_{\delta,\lambda}\res K+X)\,.\]
Since $\tilde F_x(T) = (1+\langle L,f(x)\rangle) F_x(T)$ this in turn implies 
\[ F_x(R_{\delta,\lambda}\res K) \leq (1+Cr) F_x(R_{\delta,\lambda}\res K+X)\,.\]
Using that the integrand is smooth and that the currents considered are supported in $B_r(x)$ one deduces that, upto choosing a bigger constant $C$, this inequality remain true even when $F_x$ is replaced by $F$. From there the almost minimality follows easily.\\
Using Lemma \ref{l:hausdorffconvergence} one again concludes that, for $\delta$ small and $\lambda$ large enough,  $R_{\delta,\lambda}$ are contained in arbitrarily small tubular neighbourhoods of $\Sigma$ and satisfy $G(R_{\delta,\lambda})\leq G(\Sigma)$, contradicting  Theorem 3 of \cite{Wh}. Repeating the steps of section \ref{s:quant} yields Theorem \ref{th1_w} with $F$ replaced by the above $G$. Theorem \ref{th2} can then be proved exactly as in \cite{Wh}, but instead of using Theorem 3 of \cite{Wh} we use our quantitative Theorem \ref{th1_w} for $G$ as above.

%
%

%
%

\bibliographystyle{plain}

%
%

\vskip .5 cm

{\parindent = 0 pt\begin{footnotesize}

D.I. \& A.M.
\\
Institut f\"ur Mathematik
Mathematisch-naturwissenschaftliche Fakult\"at
Universit\"at Z\"urich\\
Winterthurerstrasse 190, CH-8057 Z\"urich\\

e-mail D.I.: {\tt dominik.inauen@math.uzh.ch}\\
e-mail A.M: {\tt andrea.marchese@math.uzh.ch}

\end{footnotesize}
}

\end{document}